\documentclass[12pt]{article}
\usepackage{amsmath, amssymb, amsthm}   
\usepackage{graphicx}   
\usepackage{verbatim}   
\usepackage{color}      

\usepackage{ulem, utopia}
\usepackage[top=2.5cm, bottom=2.5cm, left=2.5cm, right=2cm]{geometry}
\theoremstyle{definition}
\newtheorem{definition}{Definition}

\theoremstyle{theorem}

\newtheorem{lemma}[definition]{Lemma}
\newtheorem{theorem}[definition]{Theorem}
\newtheorem{corollary}[definition]{Corollary}

\newcommand{\pr}{\mathbb{P}}
\newcommand{\bE}{\mathbb{E}}

\newcommand{\1}{{\bf 1}}

\begin{document}
	
	\title{Limit theorem for perturbed  random walks} 
	
	\author{Hoang-Long Ngo\footnote{Hanoi National University of Education. 136 Xuan Thuy, Cau Giay, Hanoi, Vietnam.   	
	Email: ngolong@hnue.edu.vn}         \qquad 
		Marc Peign\'e \footnote{Institut Denis Poisson, University of Tours. Parc de Grandmont
			37200 Tours, France.    
			Email: peigne@univ-tours.fr}}

\maketitle 	
	
	\textbf{Abstract } 
	 We consider random walks perturbed at zero which behave like (possibly different) random walks with i.i.d. increments on each half lines and restarts at $0$ whenever they cross that point.  We show that the perturbed random walk, after being rescaled in a proper way, converges to a skew Brownian motion whose parameter is defined by renewal functions of the simple random walks and the transition probabilities from $0$.   

		 

  \section{Introduction}

 Let $(S(n))_{n \geq 0}$ be a  random walks on $\mathbb Z$ starting from $0$ whose increments $\xi_k, k \geq 1$, are i.i.d  with finite second moment.  A continuous time process $(S(t))_{t\geq 0}$  can be constructed  from the sequence $(S(n))_{n\geq 0}$ by using the linear interpolation between the values at integer points. According to the well-known Donsker's theorem,  the sequence of stochastic processes $(S_n(t))_{n \geq 1}$, defined by
\[
S_n(t):= {1\over \sqrt{n}}S(nt), \quad n \geq 0,
\]
 weakly converges  in the space continuous functions  on $[0, 1]$ to the Wiener process,  as $n \to +\infty$.
   
 In this paper, we  consider the more general cases  of spatially inhomogeneous random walks  $(X(n))_{n \geq 0}$  on $\mathbb Z$   which model some discrete  time diffusion in a one dimensional space with two different media $\mathbb Z^-$ and $\mathbb Z^+$ and a  barrier $\{0\}$. Let us state the main result of  the present article. 
\begin{theorem}\label{skew}
Assume that 
\begin{enumerate} 
 \item $(\xi_n)_{n \geq 1}$  and  $(\xi'_n)_{n \geq 1}$ are sequences of $\mathbb Z$-valued, centered and   i.i.d. random variables with finite second moments     and   respective variances $\sigma^2$ and $\sigma'^2$;
\item   the supports of the distribution of the $\xi_n$ and $\xi'_n$ are not included in the coset of a proper subgroup of $\mathbb Z$ (aperiodicity condition); 
\item $(\eta_n)_{n \geq 1}$  is a sequence of  $\mathbb Z$-valued and i.i.d. random variables  such that  $\bE[\vert \eta_n\vert ] <+\infty$, and $\pr[\eta_n =0] < 1$;  
\item the sequences  $(\xi_n)_{n \geq 1}$, $(\xi'_n)_{n \geq 1}$ and $(\eta_n)_{n \geq 1}$ are independent.
\end{enumerate}
Let $(X(n))_{n \geq 0}$  be the $\mathbb Z$-valued process defined by:
$X(0)=0$ and, for $n \geq 1$,  
\begin{equation}\label{X(n)}
X(n)= \begin{cases}
X(n-1) + \xi_n & \ \ \text{ if }  \quad  X{(n-1)} > 0 \quad \text{and} \quad X{(n-1)} + \xi_n> 0, \\
0 & \ \ \text{ if }  \quad  X{(n-1)} > 0 \quad \text{and} \quad X{(n-1)} + \xi_n\leq 0, \\
\eta_{n} & \ \ \text{ if }  \quad  X{(n-1)} = 0, \\
X(n-1) + \xi'_n & \ \ \text{ if }  \quad  X{(n-1)} < 0 \quad \text{and} \quad X{(n-1)} + \xi'_n< 0, \\
0& \ \ \text{ if }  \quad  X{(n-1)} < 0 \quad \text{and} \quad X{(n-1)} + \xi'_n\geq 0. 
\end{cases}
\end{equation}
Let  $(X(t))_{t\geq0}$ be the continuous time process   constructed  from the sequence $(X(n))_{n \geq 0}$ by  linear interpolation between the values at integer points, and, for any $n \geq 1, t \geq 0$, set 
\[
X_n(t):=  \begin{cases} \displaystyle {1\over \sigma\sqrt{n}}X(nt) & \text{\it when} \quad X(nt)\geq 0,\\
\displaystyle {1\over \sigma'\sqrt{n}}X(nt) & \text{\it when} \quad X(nt)\leq 0. 
\end{cases}
\]
Then,  as $n \to +\infty$,  the sequence of stochastic processes $(X_n(t))_{n \geq 1}$ 
 weakly converges  in the space of continuous functions  on $[0, +\infty)$ to the skew Brownian motion $( B^\alpha(t))_{t \geq 0}$   on $\mathbb R$ with     parameter $\alpha$  depending on the distribution of the $\xi_n, \xi'_n$ and $\eta_n$   as follows:
  \begin{align} \label{def:alpha} 
\alpha =   {c_1 \bE[h(\eta_1)\1_{\{\eta_1>0\}}]\over  c_1 \bE[h(\eta_1)\1_{\{\eta_1>0\}}] + c'_1 \bE[h'(-\eta_1)\1_{\{\eta_1<0\}}] }, 
\end{align}  
 where
 \begin{enumerate}
 \item $h$   is  the ``descending renewal function''  \footnote{We refer to sections \ref{sec:2.1}  and \ref{proofoftheorem} for the definition  of the``renewal function'' of an oscillating random walk.} of the random walk    $
  (S(n))_{n \geq 0}$ with increments $\xi_k$; 
  \item  $ \displaystyle  c_1={1\over \sqrt{\pi}}\exp\left(\sum_{k=1}^{+\infty}{1\over k}\left(\mathbb P[S(k) \textcolor{black}{\geq} 0]-{1\over 2}\right)\right)$;
  \item   $h'$  is  the   ``ascending renewal function'' of the random walk  $
  (S'(n))_{n \geq 0}$ with increments $\xi'_k$;
  \item $\displaystyle c'_1={1\over \sqrt{\pi}}\exp\left(\sum_{k=1}^{+\infty}{1\over k}\left(\mathbb P[S'(k) \textcolor{black}{\leq} 0]-{1\over 2}\right)\right)$.
 \end{enumerate} 
 \end{theorem}

  The Markov chain  $(X(n))_{n \geq 0}$ has been the object of several studies in various contexts. 
  
  It first appeared in  \cite{HarrisonShepp}  where the skew Brownian motion (SBM) on $\mathbb R$ is obtained as the weak limit of a normalized  simple random walk   on $\mathbb Z$ which has special behavior only  at the origin; in this seminal work, the authors consider the case when the $\xi_n$ and $\xi'_n$  are Bernoulli symmetric random variables and  the variables $\eta_k$ are also $\{-1, 1\}$ valued, but  with respective probabilities $1-\alpha$ and $\alpha$.  The key step of the proof  in \cite{HarrisonShepp} relies on the reflection principle  
  (but   the convergence of the  finite dimensional distributions is not done in details), which is  valid only for Bernoulli symmetric random walks. Let us cite also the long  and rich article  \cite{Lejay2006} by A. Lejay on the construction of the skew Brownian motion, where    a slightly different argument  is also presented:  the  approximation of  the SBM$(\alpha)$ by Bernoulli symmetric random walks  is obtained starting from a trajectory of the  SBM($\alpha$)   and constructing recursively a suitable sequence of stopping times. Both proofs only work for Bernoulli symmetric random walks. Note also that Harrison and Shepp mentioned without a proof in \cite{HarrisonShepp} that such a result holds for  arbitrary integrable random variables $\eta_n$; Theorem \ref{skew} covers (and extends) this case.   
    
  More recently, in \cite{PilipenkoPryhod'ko2012}, the SBM appears as the weak  limit  of  a random walk in $\mathbb Z$ whose transition probabilities coincide with those of a symmetric random walk with unit steps throughout except for a fixed neighborhood $\{-m,\ldots, m\}$ of zero, called a ``membrane''. Assuming that, from the membrane, the chain jumps to an arbitrary  point of the set $\{-m-1,\ldots, m+1\}$, the authors  describes the different possible weak limit of the suitably scale process according to the fact that the sites $-m-1$ and $m+1$ are essential or not. All possible limits for the corresponding random walks are described. In \cite{IksanovPilipenko2016}, it is proved that this convergence holds in fact  for random walks such that   the absolute value  $\vert \xi_k\vert $  of the steps outside the membrane  are  bounded by $2m+1$; this last condition ensures that the random walk cannot jump over the membrane without passing through it.  A. Iksanov and A. Pilipenko offer in \cite{IksanovPilipenko2016}   a new approach which is based on the martingale characterization of the SBM.  
  In the present paper, the membrane is reduced to $\{0\}$ and the  assumption of ``bounded jumps out of the membrane''  is replaced  by  the absorption condition at $0$; nevertheless, Theorem \ref{skew} may be   extended  to a finite  membrane,  with similar conditions as in \cite{PilipenkoPryhod'ko2012}.
  
  Let us  mention  that the expression of the parameter $\alpha$  is   different of the one proposed in  \cite{PilipenkoPryhod'ko2012}; it takes into account  in an explicit form the inhomogeneity of the random walk on $\mathbb Z$. In the case when the two random walks  are Bernoulli symmetric
 \footnote{The Bernoulli symmetric random walk is not aperiodic and  thus does not fall exactly within the scope of Theorem \ref{skew}; nevertheless, statements of Lemmas \ref{LemA} and \ref{LemA'} still hold in this case and the rest of our proof works. The general non aperiodic case remains an open question; in particular, as far as we know,   Lemmas \ref{LemA} and \ref{LemA'}  are still unknown in the non aperiodic case, except for particular distributions.}, it holds $h(-x)=h'(x)=x$  for any $x \geq 0$ and  $c_1=c'_1$; thus, $\alpha= {\mathbb E[(\eta_1)^+]\over  \mathbb E[\vert \eta_1\vert]}$ as announced in \cite{HarrisonShepp}.

When $\mathcal L(\xi_1)=\mathcal L(-\xi'_1)$ and $ \mathcal \eta_1$ is symmetric  (ie $\mathcal L(\eta_1)=\mathcal L(-\eta_1)$), the process $(\vert X(n)\vert )_{n \geq 0}$ above coincides with  the process $(Y(n))_{n \geq 0}$ defined by $Y(0)=0$ and, for $n \geq 1$,  
\begin{equation}\label{Y(n)}
Y(n)= \begin{cases} Y({n-1}) + \xi_n & \ \ \text{ if }  \quad  Y({n-1}) > 0 \quad \text{and} \quad Y({n-1}) + \xi_n> 0,
\\
0& \ \ \text{ if }  \quad  Y({n-1}) > 0 \quad \text{and} \quad Y({n-1}) + \xi_n\leq 0,
\\
\gamma_{n} & \ \ \text{ if }  \quad  Y({n-1}) = 0\end{cases}
\end{equation}
where $(\gamma_k)_{k \geq 1}$ is a sequence of i.i.d. $\mathbb N$-valued random variables, independent on the sequence $(\xi_n)_{n \geq 1}$.
 This  process $(Y(n))_{n \geq 0}$   is  a variation of the  so-called {\it Lindley process} $(L(n))_{n \geq 0}$, which appears in queuing theory and  corresponds to the case when the random variables $\xi_n$ and $\gamma_n$ are equal.  Recall that the Lindley process is a fundamental example of  processes governed by iterated functions systems (\cite{DF}, \cite{PW});  indeed, setting $f_a(x):= \max(x+a, 0)$ for any $a, x \in \mathbb R$, the quantity $L(n)$ equals (recall $\xi_n=\gamma_n$ for the Lindely process)
\[
L(n)= f_{\xi_n}\circ \cdots \circ f_{\xi_2}\circ f_{\xi_1}(0).
\]
 The composition of the maps $f_a$ is not commutative, this introduces several difficulties.   Namely, between two consecutive visits of $0$, the Lindley process   behaves like  a classical random walk  $(S(n))_{n \geq 1}$; thus,  its study relies on fluctuations of random walks on $\mathbb Z$. In the case of $(Y(n))_{n \geq 0}$,  the fact that after each visit of $0$ the increment is governed by another distribution, the one of the $\gamma_k$, introduces some kind  of inhomogeneity we have to control.

\begin{corollary}\label{meander}
Assume that 
\begin{enumerate} 
 \item $(\xi_n)_{n \geq 1}$ is a sequence of $\mathbb Z$-valued   i.i.d. random variables, and  $\bE[\xi_n] = 0$ and \textcolor{black}{$\bE[\xi_n^2]=\sigma^2 < \infty$};
 \item the support of the distribution of  $\xi_n$ is not included in the coset of a proper subgroup of $\mathbb Z$;
\item $(\gamma_n)_{n \geq 1}$  is a sequence of $\mathbb N$-valued and i.i.d. random variables  such that  $\bE[\gamma_n] <+\infty$, and $\pr[\gamma_n =0] < 1$;  
\item the sequences  $(\xi_n)_{n \geq 1}$ and $(\gamma_n)_{n \geq 1}$ are independent.
\end{enumerate}
Let $(Y(t))_{t\geq0}$ be the continuous time process   constructed  from the sequence $(Y(n))_{n \geq 0}$ by  linear interpolation between the values at integer points. Then,  as $n \to +\infty$,  the sequence of stochastic processes $(Y_n(t))_{n \geq 1}$, defined by

\[
Y_n(t):= {1\over \textcolor{black}{\sigma} \sqrt{n}}Y(nt), \quad n \geq 1, 0\leq t\leq 1, 
\]
 weakly converges  in the space of continuous functions  on $[0, 1]$ to the absolute value $(\vert B(t)\vert)_{t \geq 0}$ of the Brownian motion on $\mathbb R$.
 \end{corollary}

 This process has been already considered in the literature.  Let us cite for instance \cite{PilipenkoPryhod'ko2014} where the jumps $\xi_k$ are bounded from below  by $-1$ but the $\gamma_k$ are $\mathbb N$-valued with distribution in  the domain of attraction of some stable distribution (possible without first moment).

  We follow the classical strategy  to prove  invariance principles \cite{Billingsley}. Firstly,  we show that the finite dimensional distribution of the processes $(X_n)_{n \geq 0}$ and  $(Y_n)_{n \geq 0}$ do converge to the suitable limit, then the tightness of these sequences of  processes. This ``pedestrian'' approach is of interest as soon as we have a precise control of both the fluctuations of the random walks on each  half line $\mathbb Z^-$ and $\mathbb Z^+$ and the steps starting from $0$; in particular, it is quite flexible to study processes whose trajectories cannot be decomposed exactly  into independent pieces, as for instance the reflected random walk on $\mathbb Z^+$ with elastic reflections at $0$.

We feel it is easier to prove Corollary \ref{meander}, in which case there is only the  random walks $(S(n))_{n \geq 0}$  to consider; this reduces considerably the notations. Thus, the paper is organized as follows: in Section \ref{notationsandauxiliaryestimates} we recall some classical results on fluctuations of random walks and their consequences, Section \ref{proofofcorollary} is devoted to the proof of Corollary \ref{meander} and  in  Section \ref{proofoftheorem} we explain how to adapt this argument to prove the main theorem.

\section{Notations and  auxiliary estimates}\label{notationsandauxiliaryestimates}
\subsection{Notations} 
 Firstly, let us  fix notations. We denote $S=(S(n))_{n \geq 0}$ the classical random walks with steps $\xi_k$ defined by $S(0)=0$ and $S(n)= \xi_1+\ldots \xi _n$ for any $n \geq 1$ and  by $\mathcal F_n$ the $\sigma$-field generated by the random variables $\xi_1,  \gamma_1, \ldots, \xi_n,   \gamma_n$ for $n \geq 1$; by convention $\mathcal F_0= \{\emptyset, \Omega\}$. 
 
Since the random variables $\xi_k$   are centered,  the process $ (Y(n))_{n \geq 0}$   visits $0$ infinitely often; thus,   we introduce the following  sequences $(\tau^Y_l)_{l \geq 0}$  of stopping times with respect to the filtration $(\mathcal F_n)_{n \geq 0}$ defined by $ \tau^Y_0 = 0$ and, for any $l \geq 0$, 
	\[\tau^Y_{l+1} = \inf \{ n > \tau^Y_l \mid  Y(n) =  0\}.
	\]
In order to establish an invariance principle for  the process $(Y(n))_{n \geq 0}$, we  control the excursions of this process between  two visits of $0$; after each visit, the  first transition  jump  has distribution $\mathcal L(\gamma_1)$, it is thus natural to introduce the following  random variables $U_{k, n}, 0\leq k\leq n, $ defined by
\[
U_{k, n}:= \gamma_{k+1}+ \xi_{k+2}  + \ldots + \xi_n.
\]
Furthermore, the  excursions of $(Y(n))_{n \geq 0}$  between  two visits of $0$
  coincide with   some parts of the trajectory of  $(S(n))_{n \geq 0}$, thus their study is related to   the fluctuations of this random walk. Hence, we need also to introduce the sequence of \textcolor{black}{strictly} descending ladder epochs $(\ell_l )_{l \geq 0}$  of the random walk $S=(S(n))_{n \geq 0}$ defined inductively by $\ell_0 =0$ and, for any $l \geq 1$,
\[
\ell_{l+1}  := \min\{n>\ell_l \mid S(n) \textcolor{black}{<} S{(\ell_l) }\}.
\]
The random variables $\ell_{1}  ,\ell_{2}   -\ell_{1}, \ell_{3}  -\ell_{2},\ldots  $ are $\mathbb P$-a.s. finite and i.i.d.. The same property holds for the random variables
$S{(\ell_{1})  }, S{(\ell_{2} ) }-S{(\ell_{1})  }, S{(\ell_{3})  }-S{(\ell_{2})},\ldots $; in other words,  the processes  $(\ell_l )_{l\geq 0}$ and $(S{(\ell_l )})_{l \geq 0}$  are random walks on $\mathbb N$  and  $\mathbb Z$  with respective distribution $\mathcal L(\ell_1)$  and $\mathcal L(S{(\ell_1) }  )$.

We denote by $p_t(x) = \frac{1}{\sqrt{2\pi t}}e^{-x^2/2t}$ the Gaussian denstiy. The transition density of a skew Brownian motion with parameter $\alpha \in [0,1]$ is given by
\begin{align}
&p^\alpha_t(x,y) = \1_{\{x= 0\}}\left[ 2\alpha p_t(y)\1_{\{y>0\}} + 2(1-\alpha)p_t(y)\1_{\{y<0\}}\right] \notag \\
&+\1_{\{x>0\}} \left[ \left( p_t(y-x) + (2\alpha - 1)p_t(x+y)\right)\1_{\{y>0\}} + 2(1-\alpha)p_t(y-x) \1_{\{y<0\}} \right] \notag \\
& + \1_{\{x<0\}}\left[ \left( p_t(y-x) + (1-2\alpha)p_t(y+x)\right)\1_{\{y<0\}} + 2\alpha p_t(y-x)\1_{\{y>0\}}\right],
\label{skewdensity} 
\end{align}
(see \cite{RY}, page 87).

At last, in order to simply the text, we use the following notations.
Let $u=(u_n)_{n \geq 0}$ and $v=(v_n)_{n \geq 0}$ be two sequences of positive reals; we write

$\bullet\quad $ {\it
$u\stackrel{c}{\preceq}v $} (or simply $u \preceq v $) when $u_n \leq c v_n$ for some constant $c>0$ and $n$ large enough;

$\bullet\quad $ $u_n \sim v_n$ if $\lim_{n\to \infty} \frac{u_n}{v_n} = 1$.  

$\bullet\quad $ $u_n \approx v_n$ if $\lim_{n\to \infty} (u_n - v_n) = 0$.   
 
\subsection{On the fluctuation of random walks} \label{sec:2.1} 

Let $h$ be the Green function of the random walk $ (S{(\ell_l) } )_{l \geq 0}$, called sometimes  the ``descending renewal function''   of $S$, defined by
\begin{equation} \label{def:h}
h(x) = \begin{cases} 1+ \displaystyle \sum_{l=1}^{+\infty} \pr[S(\ell_l) \geq  -x] & \text{if } \quad x \geq 0,\\
0 & \text{otherwise.}\end{cases}
\end{equation} 

The function  $h$ is  harmonic for the random walk $(S(n))_{n \geq 0}$  killed when it reaches  the negative half line $(-\infty; 0]$; namely, for any $x \geq 0,$
\[
\bE[h(x+\xi_1); x+\xi_1>0  ] = h(x).
\]
This holds for any oscillating random walk, possible without finite second moment. 
 
Moreover,  the  function $h$ is  increasing, $h(0)=1$ and $h(x) = O(x)$ as $x \to \infty$ (see \cite{AGKV05}, p. 648).
 
We have also to take into account that the random walk $S$ may start from another point than the origin; hence,  for any $x \geq 0$, we set $\tau^S(x):= \inf\{ n \geq 1: x+S(n)\leq 0\}$;  it holds 
\[
[\tau^S(x)>n]= [L_n> -x]
\]
where    $L_n=\min(S(1), \ldots, S(n))$.

The following  result is a consequence of Theorem A in \cite{Kozlov} and Theorem II.6 in \cite{LePagePeigne1}. Recall that $c_1$ was defined in Theorem \ref{skew}. 
\begin{lemma} \label{LemA}
For any $x  \geq 0$, it holds that 
\begin{equation}\label{LemA1}
\pr[\tau^S(x)> n]  \sim  c_1 \frac{h(x)}{\sqrt{n }} \qquad as \quad n \to +\infty. 
\end{equation}
Moreover, there exists  a constant  $C_1 > 0$ such that for any $x \geq 0$ and $n \geq 1,$
\begin{equation}\label{LemA2}
\pr[\tau^S(x)> n] \leq C_1\frac{h(x)}{\sqrt{n}}.
\end{equation}
\end{lemma}

 Let us emphasize that, despite appearances,  inequality (\ref{LemA2}) is more difficult to  obtain than equivalence \eqref{LemA1}.  
 Let denote by $\tilde h$ the ascending renewal function of $S$, that is also the descending renewal function of the random walk $\tilde S = -S$ defined as in \eqref{def:h}.
 The following statement holds.
\begin{lemma} \label{LemA'}
There exist   constants  $c_2, C_2>0$ such that,  for any $x, y \geq 0$, 
\begin{equation}\label{LemA'1}   
\pr[\tau^S(x)> n, x+S(n) = y]  \sim  c_2 \frac{h(x)\tilde h(y)}{n^{3/2}} \qquad {\it as} \quad n \to +\infty, 
\end{equation}
and, for any $n \geq 1,$
\begin{equation}\label{LemA'2}
\pr[\tau^S(x)> n,  x+S(n)=y] \leq C_2\frac{h(x)\tilde h(y)}{n^{3/2}}.
\end{equation}
\end{lemma}
The constants $c_1$ and $c_2$ are closely connected, see (\ref{c1c2}) below.

\begin{proof}
 The first assertion  follows from Proposition 11 in \cite{Doney12}   (see also Theorem II.7  in  \cite{LePagePeigne1}).  
Inequality (\ref{LemA'2}) may be deduced from Lemma  \ref{LemA}   following  the strategy of  Proposition 2.3 in \cite{ABKV}. Let us detail this argument.

For any $n\geq 1$   and  $1\leq k \leq n$, set
$\widetilde{S}_k= -( \xi_n+\ldots + \xi_{n-k+1})$ and $\widetilde{L}_k=\min(\widetilde{S}_1, \ldots, \widetilde{S}_k)$. 
We set $A_{n, x}= [L_{[n/3]}>-x]$, $\widetilde{A}_{n, y}=[\widetilde{L}_{[n/3]}> -y]$ and  $\overline{A}_{n, x, y}=[ x+S(n)= y ]$.
It holds
\begin{align*}
[\tau^S(x)> n, x+S(n) = y] &=
[ S(1) > -x, \ldots, 
 S(n)>-x,  x+S(n)= y]\\
 &\subset A_{n, x}\cap \widetilde{A}_{n, y}\cap \overline{A}_{n, x, y}.
\end{align*}
On the one hand, the events  $A_{n, x}$ and $\widetilde{A}_{n, y}$ are independent and measurable with respect to the $\sigma$-field  $\mathcal A_n$ generated by $\xi_1, \ldots, \xi_{[n/3]}$ and $\xi_{[2n/3]+1}, \ldots, \xi_n$. On the other hand, the random variable $T_n =S(n)+S([n/3])-S([2n/3])$ is  $\mathcal A_n$-measurable and the event $\overline{A}_{n, x, y}$ may be rewritten as
\[\overline{A}_{n, x, y}=[ S([2n/3])-S([n/3])= -x+  y -T_n].
\] 
By the classical local limit theorem in $\mathbb R$,  there exists a constant $C>0$ such that
\[
\mathbb P[\overline{A}_{n, x, y}\mid \mathcal A_n]\leq \sup _{z\in \mathbb R}\mathbb P[ S([2n/3]-[n/3])=z]\leq  {C \over \sqrt{n}}.
\]
Moreover, by Lemma \ref{LemA},  
$$\mathbb P[A_{n, x}]=\mathbb P[\tau^S(x)>[n/3]]\leq 2C_1{  h(x)\over \sqrt{n}},$$
and 
$$\mathbb P[\widetilde{A}_{n, y}]=\mathbb P[\tau^{\tilde S}(y)>[2n/3]-[n/3]]
\leq  2C_1{  \tilde h(y)\over \sqrt{n}}.$$
Inequality (\ref{LemA'2}) follows with $C_2=  4C C_1^2$.
\end{proof}
Lemma \ref{LemA'} is the key point to control the behavior as $n \to +\infty$ of the probability of the events $[\tau^S(x) = n], x \geq 0$.  The following result is more precise than Lemma \ref{LemA}.
\begin{lemma}\label{LemA''}
	 For any $x  \geq 0$,  as $n \to +\infty$, 
	$$ \pr[\tau^S(x) = n] \sim  {c_1 \over 2 }   \ h(x)\  {1\over  n^{3/2}}  \qquad {\it as} \quad n \to +\infty, $$
	and  there exists   a constant   $    C_3>0$ such that,  for any  $x \geq 0$ and $n\geq 1$,
		\[
		\pr[\tau^S(x) = n] \leq  C_3\frac{h({x})}{n^{3/2}}.
		\]
\end{lemma}
\begin{proof}
	We write 
	\begin{align*}
	\pr[\tau^S(x) = n] &= \sum_{y\geq 1}  \pr[\tau^S(x) > n-1, x+S(n-1)  = y, y + \xi_n \leq 0]\\
	&= \sum_{y\geq 1}  \pr[\tau^S(x) > n-1, x+S(n-1)  = y] \pr[y + \xi_1 \leq 0].
	\end{align*}
	From   Lemma \ref{LemA'}  and   the fact that $\tilde h(y) =O(y)$ as $y \to +\infty$,    for any $y \geq 1$,
 	\begin{align*}
 \displaystyle n^{3/2}&  \pr[\tau^S(x) > n-1, x+S(n-1)= y]\pr[ \xi_1 \leq -y]  \\
	&\preceq \  h(x) \ \tilde h(y)\ \pr[ y+\xi_1 \leq 0]\\
	&\preceq  \   h(x) \   (1+y)\  \pr[ y+\xi_1 \leq 0], 
	\end{align*}	
with $\displaystyle \sum_{y \geq 1}(1+y)\  \pr[ y+\xi_1 \leq 0]<  +\infty$ since $\mathbb E(\xi_1^2)<+\infty$. 

Thus, the second assertion of the Lemma holds;  with  (\ref{LemA'2}) and using the dominated convergence theorem,  it yields
	\begin{align*}
	&\lim_{n\to +\infty} n^{3/2}\pr[\tau^S(x) = n]\\
	 &= \lim_{n\to +\infty} n^{3/2}\sum_{y\geq 1}  \pr[\tau^S(x) > n-1, S(n-1)+ x = y] \pr[y+ \xi_1 \leq 0 ]\\
	& = c_2 h(x) \sum_{y\geq 1}  \tilde h(y)\pr[ y+\xi_1 \leq 0].
	\end{align*}
	 Comparing with   Lemma  \ref{LemA}, the constants $c_1$ and $c_2$ satisfy the equality
\begin{equation} \label{c1c2}
c_1= 2c_2\sum_{y\geq 1} \tilde h(y) \mathbb P[y+\xi_1\leq 0] 
\end{equation}
 and the first assertion of the Lemma follows. 

\end{proof}

\subsection{On the fluctuation of the perturbed random walk $Y$}\label{sec4.1}
 
  Firstly,  notice that the sequence $(\tau^Y_l)_{l \geq 0}$ is a strictly increasing random walk on $\mathbb N$, with i.i.d. increments distributed as $\tau^Y_1$. Thus, its potential $\displaystyle \sum_{l=1}^{+\infty} \pr[\tau^Y_l = n] $ is finite for any $n \geq 0$.  In this subsection, we prove that  the distributions of $\tau^S(x)$ and $\tau^Y_1$ have the same behavior at infinity and  we describe the behavior as $n \to +\infty$ of the potential of the random walk $(\tau^Y_l)_{l \geq 0}$. 
\begin{lemma} \label{LemB} It holds, as $n \to +\infty, $
\[\pr[\tau^Y_1>n] \sim   c_1{  \bE[h(\gamma_1)\1_{\{\gamma_1>0\}}]\over \sqrt{n}}. 
\]
More precisely, 
\[
	 \pr[\tau^Y_1 = n] \sim     {c_1\over 2} \   \mathbb E[h(\gamma_1)\1_{\{\gamma_1>0\}}]  \  {1\over n^{3/2}}. 
	\]
	\end{lemma}
\begin{proof} Firstly 
\begin{align*}
\pr[\tau^Y_1 > n] &= \pr[\gamma_1>0, \gamma_1+\xi_2> 0, \ldots, \gamma_1+\xi_2+\ldots +\xi_n  >0]  
\\
&=\sum_{x > 0} \pr[x+\xi_2  > 0, \ldots, x+\xi_2+\ldots +\xi_n  >0] \mathbb P[\gamma_1=x]
\\
&=\sum_{x> 0}\pr[x+S(1)  > 0, \ldots, x+S(n-1)   >0]\mathbb P[\gamma_1=x]
\\
&=\sum_{x > 0}\pr[\tau^S(x)>n-1]\mathbb P[\gamma_1=x].
\end{align*}
Recall that $h(x) = O(x)$ as $x\to +\infty$ and $\bE[\gamma_1] < +\infty$; then, by Lemma  \ref{LemA}, it holds
\[
\sum_{x > 0}\sup_{n \geq 1} \sqrt{n}\pr[\tau^S(x)>n-1] \mathbb P[\gamma_1=x] < 2C_1 \bE[h(\gamma_1)] <+\infty,
\]
and  Lebesgue's dominated convergence theorem yields
\[
\lim_{n\to +\infty} \sqrt{n}\pr[\tau^Y_1>n] =   c_1 \sum_{x > 0}h(x)\mathbb P[\gamma_1=x] =   c_1 \bE[h(\gamma_1)\1_{\{\gamma_1>0\}}].
\]
To prove the second assertion, and 
for the convenience, we overestimate $\mathbb P[\tau^Y_1=n+2]$. It holds, for any $n \geq 0$,
\begin{align*} 
&\pr[\tau^Y_1=n+2] \\
&= \pr[\gamma_1>0,U_{0,2} > 0, \ldots, U_{0,n+1}>0, U_{0,n+2}\leq 0]
\\
&= \sum_{x > 0}   \pr[\gamma_1=x]   \pr[x+S(1) > 0, \ldots, x+S(n)>0, x+S(n+1)\leq 0]
\\
&= \sum_{x > 0}   \pr[\gamma_1=x]  \sum_{y\leq -1}     \pr[\xi_{n+1}=y]  \pr[ x+S(1) > 0, \ldots, x+S(n)>0, x+S(n)\leq \vert y\vert]
\\
&= \sum_{x > 0}   \pr[\gamma_1=x]  \sum_{y\leq -1}     \pr[\xi_{1}=y] 
 \sum_{z=1}^{\vert y\vert}  \pr[ \tau^S(x)>n, x+S(n)=z].
\end{align*}
Lemma  \ref{LemA'}   and  the fact that  $\tilde h$ is increasing,   imply, for any $y \leq -1$,
\[
 \sum_{z=1}^{\vert y\vert}  \pr[ \tau^S(x)>n, x+S(n)=z] \leq C_2{h(x)\  \vert y\vert \ \tilde h(\vert y\vert )\over n^{3/2}},
\]
with   $\mathbb E[    h(  \gamma_1 )]<+\infty$ and $\mathbb E[ \vert \xi_1\vert   \tilde h( \vert  \xi_1 \vert )]<+\infty$.

Eventually, Lemma  \ref{LemA'}  and the dominated convergence theorem yield
\begin{eqnarray*}
\lim_{n \to +\infty} n^{3/2}\pr[\tau^Y_1=n+2] &=& \mathbb E[h(\gamma_1) \1_{\{\gamma_1>0 \}}]
\sum_{y\leq -1}    
 \left(\sum_{z=1}^{\vert y\vert}  \tilde h(z)\right) \pr[\xi_{1}=y] 
 \\
 &=&\mathbb E[h(\gamma_1)\1_{\{\gamma_1>0 \}}] \sum_{y\geq 1}  \tilde h(y)\pr[ y+ \xi_1 \leq  0].
\end{eqnarray*}
\end{proof}
As a direct consequence, we get the following  estimation  of the Green function of the random walk $(\tau_l)_{l \geq 0}$. 
\begin{lemma} \label{LemB'}  As $n \to +\infty$,
\[
 \Sigma_n:=  \sum_{l=0}^{+\infty} \pr[\tau^Y_l = n] \sim   \frac{1}{c_1 \pi\ \bE[h(\gamma_1)\1_{\{\gamma_1>0 \}}]}\ {1\over  \sqrt{n}}. 
\]
\end{lemma}

\begin{proof}
We    apply  Theorem B  in \cite{Doney97}   \footnote{The limit presented in Theorem B in \cite{Doney97} is not correct; in fact, Theorem B is a corollary of  Theorem 3 where, after a tedious calculous, the constant ${\pi}$ does appear. See also  \cite{AB}, page 2. } and have   to check that 
\[
\sup_{n\geq 0} \frac{n\pr[\tau^Y_1 = n]}{\pr[\tau^Y_1 >n]} < +\infty.
\]
This follows from Lemma  \ref{LemB}.
\end{proof}
\subsection{Conditional limit theorem}
We recall the following result, which is a  consequence of Lemma 2.3 in \cite{AGKV05}. The symbol $``\Rightarrow"$ means ``weak convergence".
\begin{lemma} \label{LemC2}
Let $(S_t)_{t\geq 0}$ be the continuous time process  constructed  from the sequence $(S(n))_{n\geq 0}$ by using the linear interpolation between the values at integer points.Then 
$$\mathcal{L}\Big(\left(\frac{S([nt])}{\textcolor{black}{\sigma}\sqrt{n}}\right)_{0\leq t \leq 1}| \min\{S(1), \ldots, S(n)\}\geq -x\Big) \Rightarrow \mathcal{L}(L^+) \quad \text{as } n \to +\infty,$$
where $L^+$ is the Brownian meander. 
\end{lemma}

As a consequence, we propose the following statement.
\begin{lemma}\label{LemC}
Let $\varphi: \mathbb R \to \mathbb R$ be a bounded, Lipschitz continuous function. Then, for any $t \in [0,1]$ and $x\geq 0$, it holds
$$\lim_{n\to +\infty} \bE\left[ \varphi \left( \frac{x+ S(n - [nt])}{\textcolor{black}{\sigma}\sqrt{n}}\right)\Big| \tau^S(x) > n - [nt]\right] = \int_0^{+\infty} \varphi(z\sqrt{1-t})ze^{-z^2/2}dz.$$
\end{lemma}
\begin{proof}
We first note that
\begin{align*}
\frac{1}{n} \bE[S(n)|\tau^S(x) > n] &= \frac{ \bE[S(n),\tau^S(x)> n] }{n\pr[\tau^S(x)>n]} \\
&\leq \frac{\sqrt{ \bE[S(n)^2]\pr[\tau^S(x) > n] }}{n\pr[\tau^S(x)>n]} = \frac{\textcolor{black}{\sigma}}{ \sqrt{n\pr[\tau^S(x)>n]}}. 
\end{align*}
Applying Lemma \ref{LemA}, we have
\begin{equation}\label{eqnC1}
\lim_{n\to +\infty} \frac{1}{n} \bE[S(n)|\tau^S(x) > n] =0.
\end{equation}
Since $\varphi$ is  Lipschitz continuous, we may write, denoting $[\varphi] $ the Lipschitz coefficient of $\varphi$, 
\begin{align*} 
&\left|  \bE\left[ \varphi \left( \frac{x+ S(n - [nt])}{\textcolor{black}{\sigma}\sqrt{n}}\right)\ \Big| \ \tau^S(x) > n - [nt]\right] \right. \\ 
\quad & - \left.  \bE\left[ \varphi \left( \frac{S(n - [nt])}{\textcolor{black}{\sigma}\sqrt{n-[nt]}}\sqrt{1-t}\right)\Big| \tau^S(x) > n - [nt]\right]\right|\\ 
&\leq [\varphi] \  \bE \left[ \frac{x}{\textcolor{black}{\sigma}\sqrt{n}} + \frac{S(n - [nt])}{\textcolor{black}{\sigma}\sqrt{n-[nt]}}\left( \sqrt{1-\frac{[nt]}{n}}- \sqrt{1-t}\right) \ \Big| \  \tau^S(x) > n - [nt]\right] \\
&\leq [\varphi] \  \frac{(x+1)}{\textcolor{black}{\sigma}\sqrt{n}} \left(1 + \bE \left[\frac{S(n - [nt])}{\sqrt{n-[nt]}} \ \Big| \ \ \tau^S(x) > n - [nt]\right]\right)  \to 0 \quad \text{ as } n \to +\infty,
\end{align*}
which  derives from  \eqref{eqnC1}. Applying  Lemma \ref{LemC2}, we get the desired result. 
\end{proof}

\subsection{An useful identity}
We achieve this section with a classical result which we  use in the sequel.
\begin{lemma}\label{IMcK} (Page 17 in \cite{ItoMcKean}) For any $\alpha, \beta$ positive, it holds
	$$\int_0^{+\infty} \frac{1}{\sqrt{ t}} \exp \Big( -\alpha t - \frac {\beta}{t}\Big)dt = \sqrt{\frac{\pi}{\alpha}}e^{-2\sqrt{\alpha \beta}}.$$ 
\end{lemma}

\section{Proof of Corollary \ref{meander}}\label{proofofcorollary}
The proof uses  the classical argument  for weak convergence in the space $C[0, 1])$ of continuous functions on $[0, 1]$ (see \cite{Billingsley}, chapter 2). Firstly,  we prove  that the finite dimensional distributions of 
the process $(Y_n(t); 0\leq t \leq 1)$ converge weakly to those of  the absolute value $(\vert W(t)\vert)_{t \geq 0}$ of the Brownian motion on $\mathbb R$, then the tightness of   the distributions of this sequence of processes.

Throughout this paper, $\varphi, \varphi_1$ and $\varphi_2$ are  bounded and Lipschitz continuous functions on $\mathbb R$.

\subsection{One-dimensional distribution}
 In this section, we prove the following lemma concerning the weak convergence  of the one-dimensional distributions of the sequence $(Y_n(t))_{n \geq 1}$. 
\begin{lemma}\label{1-dim}
For any $t\in [0, 1]$, it holds
$$\lim_{n\to +\infty}\bE \left[ \varphi \left(Y_n(t)\right)\right] 
 =\int_0^{+\infty} \varphi(u) \frac{2e^{-u^2/2t}}{\sqrt{2\pi t}}du = \bE[\varphi(|B_t|)],$$
 where $B$ is a standard Brownian motion. 
\end{lemma}
\begin{proof}	
We fix  $t\in (0, 1)$ and  firstly decompose the expectation $\displaystyle \bE \left[ \varphi \left( \frac{Y({[nt]})}{\textcolor{black}{\sigma}\sqrt{n}}\right)\right] $ as follows: 
\begin{align*}
&\bE \left[ \varphi \left( \frac{Y({[nt]})}{\textcolor{black}{\sigma}\sqrt{n}}\right)\right] \\
&\approx  \sum_{k=0}^{[nt]-1} \sum_{l=0}^{+\infty}  \bE \left[ \varphi \left( \frac{Y({[nt]})}{\textcolor{black}{\sigma}\sqrt{n}}\right); \tau^Y_l = k, Y(k+1) > 0, \ldots, Y({[nt]}) > 0 \right]\\
&= \sum_{k=0}^{[nt]-1} \sum_{l=0}^{+\infty}  \bE \Big[ \varphi \left ( \frac{\gamma_{k+1} + \xi_{k+2} + \ldots + \xi_{[nt]}}{\textcolor{black}{\sigma}\sqrt{n}}\right); \tau^Y_l = k, \gamma_{k+1}>0,   \\
&   \qquad \qquad  \qquad \qquad  \qquad \qquad  \gamma_{k+1} + \xi_{k+2}>0, \ldots,  \gamma_{k+1} + \xi_{k+2}+ \ldots+ \xi_{[nt]} > 0 \Big].
\end{align*}
Since the random variables $\tau^Y_l$ are stopping times, the event $[\tau^Y_l=k]$ is independent of the variables $\gamma_{k+1},  \xi_{k+2},  \ldots  \xi_{[nt]}$;  furthermore $ \mathcal L(\gamma_{k+1},  \xi_{k+2},  \ldots  \xi_{[nt]})$ $= \mathcal L(\gamma_{1},  \xi_{1},  \ldots  \xi_{[nt]-k-1})$. Hence,  recalling that  $  \Sigma_k=\sum_{l=0}^{+\infty}\pr[\tau^Y_l = k]$, 
\begin{align*}
&\bE \left[ \varphi \left( \frac{Y({[nt]})}{\textcolor{black}{\sigma}\sqrt{n}}\right)\right] \\
&\approx   \sum_{k=0}^{[nt]-1}  \Sigma_k \bE \Big[ \varphi \left ( \frac{\gamma_{1} + S({[nt]}-k-1)}{\textcolor{black}{\sigma}\sqrt{n}}\right); \gamma_1>0, \gamma_{1} + S(1)>0, \ldots,  \\
&  \hskip 7.8cm \gamma_{1} + S([nt]-k-1) > 0 \Big]  \\
&=  \sum_{k=0}^{[nt]-1} \Sigma_k  \bE \left[ \bE\left( \varphi \left ( \frac{x + S([nt]-k-1)}{\textcolor{black}{\sigma}\sqrt{n}}\right); \tau^S(x) > {[nt]}-k-1\right) \Big|_{ \gamma_1=x>0}\right]  \\
&=  \sum_{k=0}^{[nt]-1}  \Sigma_k \bE \left[ \bE\left( \varphi \left ( \frac{x + S([nt]-k-1)}{\textcolor{black}{\sigma}\sqrt{n}}\right)\Big| \tau^S(x) > {[nt]}-k-1\right) \right.\\
& \qquad\qquad \qquad \qquad\qquad \qquad \qquad\qquad \qquad \left. \times \pr[\tau^S(x) > {[nt]}-k-1] \Big|_{  \gamma_1 =x>0}\right].
\end{align*}
 From now on,  we use the notation  $\bE[\phi(x)|_{\gamma_1 = x>0}] := \bE[\phi(\gamma_1)\1_{\{\gamma_1>0\}}]$ for any bounded function $\phi: \mathbb Z \to \mathbb R$.  
  By  Lemmas \ref{LemA} and \ref{LemB'},  each term in this sum is  ${\mathcal O}({1\over \sqrt{n}})$ and thus converges to $0$ as $n \to +\infty$. Thus, it is sufficient to  control the sum  above from $k=2$ to $k = [nt]-4$, writing it  as the integral  of  the function  $f_n:   (0,t) \to \mathbb{R}$  defined by:    for each $k = 2, \ldots, [nt]-4 $ and  any $s \in [\frac kn, \frac{k+1}{n})$, 
\begin{align*}
f_n(s) 
& = n \Sigma_k\bE \left[ \bE\left( \varphi \left ( \frac{x + S([nt]-k-1)}{\textcolor{black}{\sigma}\sqrt{n}}\right)\Big| \tau^S(x) > [nt]-k-1\right) \right. \\
&\qquad\qquad \qquad \qquad\qquad \qquad  \left. \times \pr[\tau^S(x) > {[nt]}-k-1] \Big|_{  \gamma_1=x>0}\right]\\
&= n \Sigma_k \bE \left[ \bE\left( \varphi \left ( \frac{x + S([nt]-[ns]-1)}{\textcolor{black}{\sigma}\sqrt{n}}\right)\Big| \tau^S(x) > [nt]-[ns]-1\right)\right. \\
&\qquad\qquad \qquad \qquad\qquad \qquad  \left. \times \pr[\tau^S(x) > {[nt]}-[ns]-1] \Big|_{  \gamma_1=x>0}\right], 
\end{align*}
and $f_n(s)= 0$ on $[0, {2\over n})$ and $[{[nt]-1\over n}, t)$. Hence,
\begin{align*}
\bE \left[ \varphi \left( \frac{Y({[nt]})}{\textcolor{black}{\sigma}\sqrt{n}}\right)\right] & = \int_0^{t} f_n(s)ds +{\mathcal O}\left({1\over \sqrt{n}}\right).
\end{align*}

Applying Lemmas \ref{LemA}, \ref{LemB'} and \ref{LemC}, we have:  for any $s\in (0,t)$,  
$$\lim_{n\to +\infty} f_n(s) = {1\over  {\pi}}\ \frac{1}{\sqrt{s(t-s)}}\int_0^{+\infty} \varphi(z\sqrt{t-s})ze^{-z^2/2}dz. $$

On the other hand, for   $2\leq [ns]\leq  [nt]-4,$ 
\begin{align*}
\vert f_n(s) \vert  &\leq C_1 \|\varphi\|_\infty \bE[h(\gamma_1)]  \frac{n}{\sqrt{[ns]([nt]-[ns]-1)}} \sqrt{[ns]}\Sigma_{[ns]}  \\
&\preceq  \frac{n}{\sqrt{[ns]([nt]-[ns]-1)}}\\
&\leq {2\over \sqrt{s(t-s)}} 
\end{align*}
(it is in the last   inequality  that we use  the fact that  $2\leq [ns]\leq  [nt]-4$).
Therefore, from Lebesgue dominated convergence theorem,  it follows  that 
\begin{align*}
\lim_{n\to +\infty}\bE \left[ \varphi \left( \frac{Y({[nt]})}{\textcolor{black}{\sigma}\sqrt{n}}\right)\right] &= \lim_{n\to +\infty} \int_0^{t}f_n(s) ds \\
&= {1\over  {\pi}}\ \int_0^t \frac{1}{\sqrt{s(t-s)}} \left(\int_0^{+\infty} \varphi(z\sqrt{t-s})ze^{-z^2/2}dz \right)ds
\end{align*}
Denote $s= ta$, we have 
\begin{align*}
&\lim_{n\to +\infty}\bE \left[ \varphi \left( \frac{Y({[nt]})}{\textcolor{black}{\sigma}\sqrt{n}}\right)\right] \\ 
 & = {1\over  {\pi}}\int_0^1 \frac{da}{\sqrt{a(1-a)}}\int_0^{+\infty} \varphi(z\sqrt{t(1-a)}) z e^{-z^2/2}dz\\
&= {1\over  {\pi}}\int_0^1 \frac{da}{\sqrt{a(1-a)}}\int_0^{+\infty} \varphi(x) \frac{x}{t(1-a)}\exp\Big( - \frac{x^2}{2t(1-a)}\Big) dx\\
&= {1\over  {\pi}}\int_0^{+\infty}\varphi(x) \frac{x}{t} dx \int_0^1 \frac{da}{\sqrt{a(1-a)^3}}   \exp\Big( - \frac{x^2}{2t(1-a)}\Big) da,	
	\end{align*}
where we change the variable $x = z\sqrt{t(1-a)}$. Using Lemma \ref{IMcK}, we get  	
\begin{align*}
\lim_{n\to +\infty}\bE \left[ \varphi \left( \frac{Y({[nt]})}{\textcolor{black}{\sigma}\sqrt{n}}\right)\right] 
 &=\int_0^{+\infty} \varphi(x) \frac{2e^{-u^2/2t}}{\sqrt{2\pi t}}du. 
\end{align*}
We achieve the proof of  Lemma \ref{1-dim} by noting that since $\varphi$ is Lipschitz continuous, it holds that 
\begin{align} \label{lagtime} 
& \left| 	\bE \left[ \varphi \left( \frac{Y({[nt]})}{\textcolor{black}{\sigma}\sqrt{n}}\right)\right] - \bE \left[ \varphi \left(Y_n(t)\right)\right] \right| \leq [\varphi] \bE \left[ \left| 	 \frac{Y({[nt]})}{\textcolor{black}{\sigma}\sqrt{n}}- Y_n(t) \right|\right] \notag \\
& \leq \frac{1}{\textcolor{black}{\sigma}\sqrt{n}} [\varphi] \bE \left[  	 |\xi_{[nt]+1}| + |\gamma_{[nt]+1}| \right] \to 0 \text{ as } n \to \infty. 
\end{align} 
\end{proof}

\subsection{Two-dimensional distribution}
 The convergence of the finite-dimensional distributions of $(Y_n(t))_{n \geq 1}$ is more delicate. We detail the argument for two-dimensional ones, the general case may be treated in a similar way. 

Let us   fix $0<s<t, n \geq 1$ and  denote 
$$\kappa =\kappa(n, s)= \min\{k>[ns]: Y(k) \textcolor{black}{ = } 0\}.$$
We write 
\begin{align*}
&\bE\left[ \varphi_1\left( \frac{Y({[ns]})}{\textcolor{black}{\sigma}\sqrt{n}}\right) \varphi_2 \left( \frac{Y([nt])}{\textcolor{black}{\sigma}\sqrt{n}}\right)\right] = A_1(n) + A_2(n),
\end{align*}
where 
\begin{align*}
A_1(n) &= \sum_{k = [ns]+1}^{[nt]} \bE\left[ \varphi_1\left( \frac{Y({[ns]})}{\textcolor{black}{\sigma}\sqrt{n}}\right) \varphi_2 \left( \frac{Y([nt])}{\textcolor{black}{\sigma}\sqrt{n}}\right)\1_{\{\kappa = k\}} \right],\\
A_2 (n) &= \bE\left[ \varphi_1\left( \frac{Y({[ns]})}{\textcolor{black}{\sigma}\sqrt{n}}\right) \varphi_2 \left( \frac{Y([nt])}{\textcolor{black}{\sigma}\sqrt{n}}\right)\1_{\{\kappa > [nt]\}} \right].
\end{align*}
 The term $A_1(n)$ deals with the trajectories of $ Y_k, 0\leq k\leq n,  $ which visit $0$ between $[ns]+1$ and $[nt]$ while $A_2 (n)$   concerns the others trajectories.
 
 \subsubsection{Estimate of $A_1(n)$} \label{Sec:3.2.1}
As in the previous section, we decompose $A_1(n)$ as   
\begin{align*}
&A_1(n) = \sum_{k_2 = [ns]+1}^{[nt]} \bE\left[ \varphi_1\left( \frac{Y({[ns]})}{\textcolor{black}{\sigma}\sqrt{n}}\right) \1_{\{\kappa = k_2\}} \right] \bE \left[\varphi_2 \left( \frac{Y([nt]-k_2)}{\textcolor{black}{\sigma}\sqrt{n}}\right)\right]
\\
&
\approx  \sum_{l=0}^{+\infty} \sum_{k_1=0}^{[ns]-1} 
\sum_{k_2 = [ns]+1}^{[nt]} \bE \left[\varphi_2 \left( \frac{Y([nt]-k_2)}{\textcolor{black}{\sigma}\sqrt{n}}\right)\right]\\
& \  \times  \bE\left[ \varphi_1\left( \frac{Y({[ns]})}{\textcolor{black}{\sigma}\sqrt{n}}\right); \tau^Y_l=k_1, \kappa = k_2; \gamma_{k_1+1}>0, U_{k_1, k_1+2}>0, \ldots, U_{k_1, [ns]}>0\right]\\
& = \sum_{l=0}^{+\infty} \sum_{k_1=0}^{[ns]-1}  \sum_{k_2 = [ns]+1}^{[nt]}   \pr[\tau^Y_l=k_1] \bE \left[\varphi_2 \left( \frac{Y([nt]-k_2)}{\textcolor{black}{\sigma}\sqrt{n}}\right)\right]\\ 
&\  \times \bE\left[ \varphi_1\left( \frac{Y({[ns]-k_1})}{\textcolor{black}{\sigma}\sqrt{n}}\right);  \gamma_1>0, U_{0, 2}>0, \ldots, U_{0,k_2-k_1-1}>0, U_{0,k_2 - k_1} \leq 0\right]\\
&= \sum_{k_1=0}^{[ns]-1}  \Sigma_{k_1} \sum_{k_2 = [ns]+1}^{[nt]}    \bE \left[\varphi_2 \left( \frac{Y([nt]-k_2)}{\textcolor{black}{\sigma}\sqrt{n}}\right)\right]\\ 
&\  \times \bE\left[ \bE\left[ \varphi_1\left( \frac{x+ S([ns]-k_1-1)}{\textcolor{black}{\sigma}\sqrt{n}}\right);   \tau^S(x) = k_2-k_1 -1  \right]\Big|_{\gamma_1 = x>0}\right].
\end{align*}
For  any $2\leq k_1< [ns]-6$ and $[ns] \leq   k_2\leq [nt]$ and any $s_1 \in  [ {k_1\over n},  {k_1+1\over n})$ and $s_2 \in  [ {k_2\over n},  {k_2+1\over n})$, we write 
\begin{align*}
f_n&(s_1, s_2) = n^2 \Sigma_{k_1} \bE \left[\varphi_2 \left( \frac{Y([nt]-[ns_2])}{\textcolor{black}{\sigma}\sqrt{n}}\right)\right]\\
& \times \bE\left[ \bE\left[ \varphi_1\left( \frac{x+ S([ns]-[ns_1]-1)}{\textcolor{black}{\sigma}\sqrt{n}}\right);   \tau^S(x) = [ns_2]-[ns_1] -1\right]\Big|_{\gamma_1 = x>0}\right],
\end{align*}
and $f_n(s_1, s_2)= 0$ for the others values of $k_1 $, such that $0\leq k_1\leq [ns]$. Hence,
\begin{align*}
A_1(n)= \int_0^s ds_1\int _s^t ds_2\  f_n(s_1, s_2)  +{\mathcal O}\left({1\over \sqrt{n}}\right).
\end{align*}

   Let us study the pointwise  limit of the sequence $(f_n)_{n \geq 1}$ as $n \to +\infty$. This is 
 here that we apply the fact that $\gamma$ and $\xi$ take integer values,  in order to apply Caravenna-Chaumont's result \cite{CC}.  We may also consider the case when $\gamma$ and $\xi$ have absolutely continuous distributions as well;  in both cases, it is possible to fix the arrival point of the random walk  $S$ at time $[ns_1]-[ns_2]-2$ and to consider its rescaled  limit as $n \to +\infty$.   We write  $k_1= [ns_1]$ and $k_2= [ns_2]$. For any $x > 0$, it holds that 
\begin{align*}
& \bE\left[ \varphi_1\left( \frac{ x+ S([ns]-k_1-1)}{\textcolor{black}{\sigma}\sqrt{n}}\right);   \tau^S(x) = k_2-k_1-1 \right]\\
&= \sum_{y=1}^{+\infty} \pr[y + \xi_{k_2-k_1-1}\leq 0]\\
&\times \bE\left[ \varphi_1\left( \frac{x+ S([ns]-k_1-1)}{\textcolor{black}{\sigma}\sqrt{n}}\right);  \tau^S(x) > k_2 - k_1 - 3,  x+S(k_2-k_1-2) = y \right] \\
&= \sum_{y=1}^{+\infty} \pr[  \tau^S(x) > k_2 - k_1 - 3, x+S(k_2-k_1-2) = y] \times  \pr[y + \xi_{1}\leq 0]\\
&\times \bE\left[ \varphi_1\left( \frac{x+ S([ns]-k_1-1)}{\textcolor{black}{\sigma}\sqrt{n}}\right)\big|  \tau^S(x) > k_2 - k_1 - 3, x+S(k_2-k_1-2) = y \right]\\
&= \sum_{y=1}^{+\infty}\pr[  \tau^S(x) > k_2 - k_1 - 3, x+S(k_2-k_1-2) = y]  \times \pr[y + \xi_{1}\leq 0]\\
&\times  \bE\left[ \varphi_1\left( \frac{x+ S([ns]-k_1-1)}{\textcolor{black}{\sigma}\sqrt{n}}\right)\big|  \tau^S(x) > k_2 - k_1 - 3, x+S(k_2-k_1-2) = y \right].
\end{align*}  
 By  Corollary 2.5 in \cite{CC}, the random walk bridge   conditioned to stay positive, starting at $x$ and ending at $y$, under linear interpolation and diffusive rescaling, converges in distribution on $C([0,1],\mathbb R) $ toward the normalized Brownian excursion; in other words, for each $x, y \geq 1$,  
\begin{align*}
&\lim_{n\to   +\infty} \bE\left[ \varphi_1\left( \frac{x+ S([ns]-k_1-1)}{\textcolor{black}{\sigma}\sqrt{n}}\right)\big|    \tau^S(x) > k_2 - k_1 - 3, x+S(k_2-k_1-2) = y \right]\\
&= \int_0^{+\infty} 2\varphi_1(u\sqrt{s_2-s_1}) \exp\left(-\frac{u^2}{2\frac{s-s_1}{s_2-s_1}\frac{s_2-s}{s_2-s_1}}\right) \frac{u^2 }{\sqrt{2\pi \frac{(s-s_1)^3}{(s_2-s_1)^3}\frac{(s_2-s)^3}{(s_2-s_1)^3}}}du \\
&= {2\over {\sqrt{2\pi}} }\int_0^{+\infty}  \varphi_1(v) \exp\left(-\frac{v^2}{2\frac{(s-s_1)(s_2-s)}{s_2-s_1}}\right) \frac{v^2  }{ \sqrt{{(s-s_1)^3(s_2-s)^3}\over (s_2-s_1)^3}}dv.
\end{align*}
  By Lemma \ref{1-dim},   $\displaystyle \bE \left[\varphi_2 \left( \frac{Y([nt]-[ns_2])}{\textcolor{black}{\sigma}\sqrt{n}}\right)\right]$ converges  to $\displaystyle \int_0^{+\infty} \varphi_2(u) \frac{2e^{-u^2/2(t-s_2)}}{\sqrt{2\pi (t-s_2)}}du.$

	To apply the dominated convergence theorem, notice that 
	the quantities \\ 
	$\bE\left[ \varphi_1\left( \frac{x+ S([ns]-k_1-1)}{\textcolor{black}{\sigma}\sqrt{n}}\right)\big|    \tau^S(x) > k_2 - k_1 - 3,   x+S(k_2- k_1-2) = y \right]$
	and    
	$ \bE \left[\varphi_2 \left( \frac{Y([nt]- k_2)}{\textcolor{black}{\sigma}\sqrt{n}}\right)\right]$ are bounded by $\vert \varphi_1\vert_\infty$ and $\vert \varphi_2\vert_\infty$ respectively.  Furthermore, on  the one hand, by Lemmas \ref{LemA''}  
	it holds, for any $x, y >0$, that 
	\begin{align*}
	&n^{3/2}\pr[ \tau^S(x) > k_2 - k_1 - 3, x+S(k_2- k_1 -2) = y]
	\\
	&=
	n^{3/2}\pr[\tau^S(x)>k_2- k_1 - 2; x+S(k_2-k_1-2) = y]
	\\
	&\leq \frac{C_2 h(x) \tilde h(y)}{(s_2 - s_1)^{3/2}},
	\end{align*}
	and 
	\begin{align*} 
	 \lim_{n\to \infty}n^{3/2}&\pr[  \tau^S(x) > k_2 - k_1 - 3, x+S(k_2-k_1-2) = y] = \frac{c_2h(x) \tilde{h}(y)}{(s_2-s_1)^{3/2}}.
	\end{align*} 
	Moreover, 
	\begin{align*}
	\lim_{n\to \infty}& n^{3/2}\sum_{y=1}^\infty  \pr[y + \xi_{1}\leq 0]  \pr[ \tau^S(x) > k_2 - k_1 - 3, x+S(k_2- k_1 -2) = y]   \\
	= \lim_{n\to \infty}& n^{3/2}\pr[  x+S(1)>0, \ldots, x+S(k_2 - k_1 - 2) >0, x+S(k_2- k_1-1) \leq 0] \\
	= \lim_{n \to \infty}& n^{3/2}\pr[\tau^S(x) = [ns_2]- [ns_1]]  = \frac{c_1}{2(s_2-s_1)^{3/2}} h(x).  
	\end{align*}
Eventually, since $\displaystyle \sum_{y \geq 1} \tilde{h}(y)\mathbb P(y+\xi_1\leq 0)<+\infty$, 	
 the Lebesgue dominated convergence theorem  and identity (\ref{c1c2})  imply that, for any $0<s_1<s_2<t$,  the sequence $(f_n)_{n \geq 1}$ converges to $f$ given by 
\begin{align*} 
f(s_1, s_2)= {1\over \pi^{2} \sqrt{ s_1}}\int_0^{+\infty}  \varphi_1(v) & \exp\left(-\frac{v^2}{2\frac{(s_2-s)(s-s_1)}{s_2-s_1}}\right) \frac{v^2  }{ \sqrt{{(s-s_1)^3(s_2-s)^3} }}dv\\
& \times \int_0^{+\infty} \varphi_2(u) \frac{ e^{-u^2/2(t-s_2)}}{\sqrt{  t-s_2}}du.
\end{align*}
By the same argument, the function $f_n$ may be dominated as follows:  for $2\leq k_1< [ns]-6$, $[ns]\leq   k_2\leq [nt]$ and $s_1 \in  [ {k_1\over n},  {k_1+1\over n})$ and $s_2 \in  [ {k_2\over n},  {k_2+1\over n})$
\begin{align*}
 \vert f_n(s_1, s_2)
 &\vert  \preceq \frac{n^2}{\sqrt{[ns_1]}([ns_2]-[ns_1]-2)^{3/2}}
 \\
 &\preceq
  {1\over \sqrt{s_1}(s_2-s_1)^{3/2}},
 \end{align*}
 with $\displaystyle \int_{0 }^s ds_1\int_{ s}^t ds_2 { 1 \over \sqrt{s_1}(s_2-s_1)^{3/2}}<+\infty.$ Hence
\begin{align*}
& \lim_{n\to +\infty} A_1(n)  
=   \int_{0 }^s ds_1\int_{ s}^t ds_2 f(s_1, s_2)
\\
&= {1\over \pi^{2} }\int_{0 }^s {ds_1\over  \sqrt{ s_1}}\int_{ s}^t ds_2
\int_0^{+\infty}  \varphi_1(v) \exp\left(-\frac{v^2}{2\frac{(s_2-s)(s-s_1)}{s_2-s_1}}\right)  \frac{v^2  }{ \sqrt{{(s-s_1)^3(s_2-s)^3}}} 
\\
& \qquad \qquad\qquad \qquad\qquad \qquad\qquad \qquad \times \int_0^{+\infty} \varphi_2(u) \frac{e^{-u^2/2(t-s_2)}}{\sqrt{  t-s_2}}dudv.
\end{align*}
Set $a = s_1/s$ and $b = (s_2-s)/(t-s)$; we obtain $\lim_{n\to +\infty} A_1(n)= A_1$ with
\begin{align*}
& A_1   = \frac{1}{\pi^{2} } \frac{1}{s( t-s)} \int_0^{+\infty} du \int_0^{+\infty} dv \int_0^1 {da\over \sqrt{a(1-a)^3}} \int_0^1 {db\over   \sqrt{b^3(1-b)}}  \varphi_1(v)\varphi_2(u)  \\
&\ \times  \exp \Big( - \frac{u^2}{2(t-s)(1-b)} \Big) 
v^2 \exp\Big( - \frac{v^2}{2(t-s)b}\Big) \exp \Big( - \frac{v^2}{2s(1-a)}\Big). 
\end{align*}
Note that  (with the change of variable $x= {a\over 1-a}$)
\begin{align*}
\int_0^1  \frac{1}{\sqrt{a(1-a)^3}}\frac{v}{\sqrt{s}} \exp \Big( - \frac{v^2}{2s(1-a)}\Big) da = \sqrt{2\pi}e^{-v^2/2s}.
\end{align*}
Now we compute 
$$I = \int_0^{1} \frac{v}{\sqrt{b^3 (1-b)}} \exp \Big( - \frac{u^2}{2(t-s)(1-b)} \Big)   \exp\Big( - \frac{v^2}{2(t-s)b}\Big)db.$$
Denote $z = \frac{b}{1-b}$, we get
\begin{align*}
I &= \exp\Big( - \frac{u^2+ v^2}{2(t-s)} \Big) \int_0^{+\infty} \frac{v}{z^{3/2}} \exp \Big(-  \frac{u^2z}{2(t-s)} - \frac{v^2}{2(t-s)z}\Big) dz.
\end{align*}
Change $z$ by $1/z$, we get 
\begin{align*}
I &= \exp\Big( - \frac{u^2+ v^2}{2(t-s)} \Big) \int_0^{+\infty} \frac{v}{z^{1/2}} \exp \Big(-  \frac{u^2}{2(t-s)z} - \frac{v^2z}{2(t-s)}\Big) dz.
\end{align*} 
By using Lemma \ref{IMcK}, we get
\begin{align*}
I &= \exp\Big( - \frac{u^2+ v^2}{2(t-s)} \Big) \sqrt{2(t-s)\pi}
\exp\Big(-\frac{uv}{t-s}\Big) \\ &=  \sqrt{2(t-s)\pi} \exp\Big( - \frac{(u+ v)^2}{2(t-s)} \Big).
\end{align*} 
Therefore, 
\begin{align} \label{lim:A1} 
 &\lim_{n \to \infty} A_1(n) = \frac{2}{\pi \sqrt{ s (t-s)}}  \int_0^{+\infty} \int_0^{+\infty} \varphi_1(v) \varphi_2(u) e^{-v^2/2s}e^{ - \frac{(u+v)^2}{2(t-s)}} dudv.
\end{align}

\subsubsection{Estimate of $A_2 (n)$} \label{Sec:3.2.2} 
	We decompose $A_2 (n)$ as 
\begin{align*}
&  \sum_{k \leq  [ns]} \sum_{l\geq 0} \bE\Big[ \varphi_1\left( \frac{Y({ [ns]})}{\textcolor{black}{\sigma}\sqrt{n}}\right) \varphi_2 \left( \frac{Y([nt])}{\textcolor{black}{\sigma}\sqrt{n}}\right);   \tau_{l} = k, \gamma_{k+1}>0,  U_{k, k+2} >0, \ldots, U_{k, [nt]} > 0\Big]\\
=&  \sum_{k \leq  [ns]} \sum_{l\geq 0} \bE\Big[ \varphi_1\left( \frac{U_{k,  [ns]}}{\textcolor{black}{\sigma}\sqrt{n}}\right) \varphi_2 \left( \frac{U_{k,[nt]}}{\textcolor{black}{\sigma}\sqrt{n}}\right);   \tau_{l} = k,  \gamma_{k+1}>0,  U_{k, k+2} >0, \ldots, U_{k, [nt]} > 0\Big]\\
=& \sum_{k \leq  [ns]} \sum_{l\geq 0} \bE\Big[ \varphi_1\left( \frac{U_{0,  [ns]-k}}{\textcolor{black}{\sigma}\sqrt{n}}\right) \varphi_2 \left( \frac{U_{0,[nt]-k}}{\textcolor{black}{\sigma}\sqrt{n}}\right);    \tau_1 > [nt]-k\Big]\pr[\tau_{l} = k]\\
\approx &  \sum_{k <  [ns]} \Sigma_k
\times \bE\Bigg[ \bE\Big[ \varphi_1\left( \frac{x + S(  [ns]-k-1)}{\textcolor{black}{\sigma}\sqrt{n}}\right) \\
&\qquad   \qquad \qquad \times  \varphi_2 \left( \frac{x + S([nt]-k-1)}{\textcolor{black}{\sigma}\sqrt{n}}\right);    \tau^S(x) > [nt]-k-1\Big]\Bigg|_{ \gamma_1=x>0}\Bigg] .
\end{align*}
For $u \in (0,s]$, we denote 
\begin{align*} g_n(u) = n&\Sigma_{[nu]} \bE\Bigg[ \bE\Big[  \varphi_1  \left( \frac{x+ S(  [ns]-[nu]-1)}{\textcolor{black}{\sigma}\sqrt{n}}\right) \\
& \times  \varphi_2 \left( \frac{x+ S([nt]-[nu]-1)}{\textcolor{black}{\sigma}\sqrt{n}}\right);    \tau^S(x) > [nt]-[nu]-1\Big]\Bigg|_{ \gamma_1=x>0}\Bigg].
\end{align*} 
By Lemmas    \ref{LemA} and \ref{LemB'}, it is clear that    
$\displaystyle 0\leq g_n(u) \preceq {1\over \sqrt{u(t-u)}}$ .

 Now, let us compute  the pointwise limit on $(0, s]$ of the sequence $(g_n)_{n \geq 1}$. We write $g_n(u) $ as
	\begin{align*}
	g_n(u) &= n \Sigma_{[nu]}\bE\Bigg[ \bE\Big[ \varphi_1\left( \frac{x+ S(  [ns]-[nu]-1)}{\textcolor{black}{\sigma}\sqrt{n}}\right) \\
	& \qquad  \times  \varphi_2 \left( \frac{x+ S(([nt]-[nu]-1)}{\textcolor{black}{\sigma}\sqrt{n}}\right);    \tau^S(x) > [nt]-[nu]-1\Big]\Bigg|_{ \gamma_1=x>0}\Bigg]\\
	&=n \Sigma_{[nu]}\bE\Bigg[ \bE\Big[ \varphi_1\left( \frac{x+ S(  [ns]-[nu]-1)}{\textcolor{black}{\sigma}\sqrt{n}}\right)  \varphi_2 \left( \frac{x+ S([nt]-[nu]-1)}{\textcolor{black}{\sigma}\sqrt{n}}\right) \Big|  \\
	& \qquad \qquad   \tau^S(x) > [nt]-[nu]-1\Big]  \pr[ \tau^S(x) > [nt]-[nu]-1] \Bigg|_{ \gamma_1=x>0}\Bigg].
	\end{align*}
	Since $\varphi_1, \varphi_2$ are bounded and  continuous on $\mathbb R$, it follows from Theorem 3.2 in \cite{Bol} and Theorems 2.23 and 3.4 in \cite{Iglehart74} that\footnote{In  \cite{Iglehart74} the author needed the third order moment of the increment is finite; in fact, it only requires finite second moment  \cite{Bol}. }
	\begin{align*}
	&\lim_{n \to +\infty} \bE\Big[ \varphi_1\left( \frac{x+ S(  [ns]-[nu]-1)}{\textcolor{black}{\sigma}\sqrt{n}}\right) \\
	&\qquad \times \varphi_2 \left( \frac{x+ S([nt]-[nu]-1)}{\textcolor{black}{\sigma}\sqrt{n}}\right)\Big|    \tau^S(x) > [nt]-[nu]-1\Big]\\
	&=\lim_{n \to +\infty} \bE\Big[ \varphi_1\left( \frac{x+ S(  [ns]-[nu]-1)}{\textcolor{black}{\sigma}\sqrt{[nt]-[nu]-1}} \frac{\sqrt{[nt]-[nu]-1}}{\sqrt{n}} \right) \\
	&  \qquad \times \varphi_2 \left( \frac{x+ S([nt]-[nu]-1)}{\textcolor{black}{\sigma}\sqrt{[nt]-[nu]-1}} \frac{\sqrt{[nt]-[nu]-1}}{\sqrt{n}}\right)\Big|    \tau^S(x) > [nt]-[nu]-1\Big]\\
	&=  \int_0^{+\infty} \int_0^{+\infty}\varphi_1(y\sqrt{t-u})\varphi_2(z\sqrt{t-u}) 
	\Big( \frac{t-u}{s-u}\Big)^{3/2}
	ye^{-\frac{  t-u }{2(s-u)}y^2} \\
	& \qquad  \qquad \qquad \times \frac{e^{-{1\over 2}  {t-u\over t-s}(z-y)^2} - e^{-{1\over 2}  {t-u\over t-s}(z+y)^2}}{\sqrt{2\pi\Big(1 - \frac{s-u}{t-u}\Big)}}dy dz\\
	&=  {1\over \sqrt{2\pi(t-s)}} \int_0^{+\infty} \int_0^{+\infty}\varphi_1(y')\varphi_2(z') 
	\frac{\sqrt{t-u}}{ (s-u)^{3/2}}
	y' 
	e^{-\frac{y'^2}{2(s-u)}}\\
	& \qquad  \qquad \qquad \times \left(e^{-\frac{(z'-y')^2}{2(t-s)}} - e^{-\frac{(z'+y')^2}{2(t-s)}} \right)dy' dz'.
	\end{align*}

	This fact together with Lemma \ref{LemA}, Lemma \ref{LemB'} and  the  Lebesgue dominated convergence theorem  implies  that the sequence $(g_n)_{n \geq 1}$ pointwise converges to $g$ with
	\begin{align*} g (u) =& {1\over\pi^{3/2} \sqrt{2  (t-s)}} {1\over \sqrt{u(s-u)^3}}
	\\
	&\times \int_0^{+\infty} \int_0^{+\infty}\varphi_1(y')\varphi_2(z') 
	y'e^{- {y'^2\over 2(s-u)}}
	\left(e^{-{(z'-y')^2\over 2(t-s)}} - e^{- {(z'+y')^2\over 2(t-s)} }\right)dy' dz'.
	\end{align*}
	Using again Lebesgue's dominated convergence theorem, we get 
	\begin{align}
	&\lim_{n\to +\infty} A_2 (n) \notag \\
	&= \lim_{n\to +\infty} \frac{1}{n} \sum_{k\leq [ns]}g_n(k/n) = \int_0^{s} g (u) du \notag \\
	&= {1\over\pi^{3/2} \sqrt{2\ (t-s)}} \int_0^s du  \int_0^{+\infty} dy'\int_0^{+\infty}dz' \notag \\
	& \qquad \qquad \times 
	\varphi_1(y')\varphi_2(z') 
	\frac{e^{-\frac{y'^2}{2(s-u)}}}{\sqrt{u(s-u)^3}}
	\frac{y'}{\sqrt{2\pi(t-s)}}
	\left(e^{-\frac{(z'-y')^2}{2(t-s)}} - e^{-\frac{(z'+y')^2}{2(t-s)}} \right) \notag \\
	&= {1\over\pi^{3/2}   s\sqrt{2\ (t-s)}}  \int_0^{+\infty} dy'\int_0^{+\infty}dz'\varphi_1(y')\varphi_2(z') 
	\left(e^{-\frac{(z'-y')^2}{2(t-s)}} - e^{-\frac{(z'+y')^2}{2(t-s)}} \right) \notag \\
	& \qquad \qquad \qquad \qquad \qquad \qquad  \times 
	\left(\int_0^1 {y' \over \sqrt{v(1-v)^3}}   e^{-\frac{y'^2}{2s(1-v)}}dv\right)  \notag \\
	&=   {1\over\pi \sqrt{  s (t-s)}}  \int_0^{+\infty} dy' \int_0^{+\infty} dz' \varphi_1(y')\varphi_2(z')  
	e^{-y'^2/2s} 
	\left(e^{-\frac{(z'-y')^2}{2(t-s)}} - e^{-\frac{(z'+y')^2}{2(t-s)}} \right).
	\label{lim:A2} 
	\end{align}
Therefore, it follows from \eqref{lim:A1} and \eqref{lim:A2} that 
$$\lim_{n \to \infty} \bE\left[ \varphi_1\left( \frac{Y({[ns]})}{\textcolor{black}{\sigma}\sqrt{n}}\right) \varphi_2 \left( \frac{Y([nt])}{\textcolor{black}{\sigma}\sqrt{n}}\right)\right]  = \bE[\varphi_1(|B_s|)\varphi_2(|B_t|)].$$
Using a similar estimate as the one in  \eqref{lagtime}, we get 
$$\lim_{n \to \infty} \bE\left[ \varphi_1\left(Y_n(s)\right) \varphi_2 \left(Y_n(t)\right)\right]  = \bE[\varphi_1(|B_s|)\varphi_2(|B_t|)],$$
which concludes the convergence of $(Y_n)$ in two-dimensional marginal distribution to a reflected Brownian motion. 
\subsection{Tightness} 
Now we verify the tightness of the sequence of processes $(Y_n(t))$. Using Theorem 7.3 in \cite{Billingsley}, it is sufficient to show that these two conditions hold
\begin{itemize}
	\item [(i)] For each positive $\gamma$, there exist an $a$ and an $n_0$ such that 
	$$\pr\left[|Y_n(0)|\geq a\right] \leq \gamma, \quad n \geq n_0.$$
	\item[(ii)] For each positive $\epsilon$ and $\gamma$, there exist a $\delta \in (0,1)$, and an $n_0$ such that 
	$$\pr\left[ w_{Y_n}(\delta) \geq \epsilon\right] \leq \gamma, \quad n \geq n_0$$
	where  $w_{Y_n}(\delta) = \sup\{|Y_n(t)- Y_n(s)|: t,s \in [0,1], |t-s| \leq \delta\}$ is the modulus of continuity of $Y_n$. 
\end{itemize}
The first condition is trivial since $Y_n(0)= Y(0) = 0$. For the second condition, we have
\begin{align*}
w_{Y_n}(\delta) 
& \leq \frac{3}{\textcolor{black}{\sigma}\sqrt{n}}\sup_{1\leq  i <  j \leq n; |i-j| \leq n\delta}|S(i) - S(j)| + \frac{1}{\textcolor{black}{\sigma}\sqrt{n}}\sup_{1\leq l \leq n}\gamma_l \1_{\{\tau^Y_l \leq n\}}.
\end{align*} 
Denote $N_n = \sum_{l=1}^\infty \1_{\{\tau^Y_l \leq n\}}$ the number of times that $Y(k), 1 \leq k \leq n$, visits zeros. We have 
\begin{align*}
&\pr \left[\frac{1}{\sqrt{n}} \sup_{1\leq l \leq n}\gamma_l \1_{\{\tau^Y_l \leq n\}} \geq \frac{\epsilon}{2} \right] \\
&= \pr \left[\sup_{1\leq l \leq N_n}\gamma_l \geq \frac{\sqrt{n}\epsilon}{2} \right]\\
&= \pr \left[\sup_{1\leq l \leq N_n}\gamma_l \geq \frac{\sqrt{n}\epsilon}{2}; N_n \leq a \sqrt{n} \right] + \pr \left[\sup_{1\leq l \leq N_n}\gamma_l \geq \frac{\sqrt{n}\epsilon}{2}; N_n > a \sqrt{n} \right]\\
&\leq  \pr \left[\sup_{1\leq l \leq a\sqrt{n}}\gamma_l \geq \frac{\sqrt{n}\epsilon}{2}\right] + \pr \left[ N_n > a \sqrt{n} \right],
\end{align*}
for any positive constant $a$. Using  the simple estimate that $(1-x)^\alpha \geq 1-\alpha x$ for any $\alpha \geq 1$ and $x \in (0,1)$, we get 
\begin{align*}
\pr \left[\sup_{1\leq l \leq a\sqrt{n}}\gamma_l \geq \frac{\sqrt{n}\epsilon}{2}\right] &= 1 - \left(1 - \pr[\gamma_1 \geq \frac{\sqrt{n}\epsilon}{2}]\right)^{[a\sqrt{n}]} \\
&\leq a\sqrt{n} \pr[\gamma_1 \geq \frac{\sqrt{n}\epsilon}{2}]\\
& \leq \frac{2a}{\epsilon}\bE[\gamma_1 \1_{\{\gamma_1 \geq \frac{\sqrt{n}\epsilon}{2}\}}].
\end{align*}
The last term tends to zeros as $n \to \infty$, since $\gamma_1$ is integrable. 

It follows from Lemma \ref{LemB'} that there exists a positive constant $C_4$ such that for all $n >0$, it holds that 
$$\sum_{l=0}^\infty \pr[\tau^Y_l = n] \leq \frac{C_4}{\sqrt{n}}.$$
Then we get
\begin{align*}
\bE[N_n] = \bE\left[ \sum_{l=0}^\infty \sum_{m=1}^n \1_{\{\tau^Y_l = m\}}\right] = \sum_{m=1}^n \sum_{l=0}^\infty \pr[\tau^Y_l = m] \leq \sum_{m=1}^n \frac{C_4}{\sqrt{m}} \leq  3C_4\sqrt{n}.
\end{align*}
This fact together with Markov's inequality implies that
$$\pr[N_n > a \sqrt{n}]	 \leq \frac{\bE[N_n]}{a\sqrt{n}} \leq \frac{3C_4}{a}.$$
Therefore, 
\begin{align*}
\limsup_{n\to \infty} \pr \left[\frac{1}{\sqrt{n}} \sup_{1\leq l \leq n}\gamma_l \1_{\{\tau^Y_l \leq n\}} \geq \frac{\epsilon}{2} \right] \leq \limsup_{n \to \infty} \frac{3C_4}{a},
\end{align*}
for any $a > 0$. It implies that 
\begin{align} \label{tight2} 
\lim_{n\to \infty} \pr \left[\frac{1}{\sqrt{n}} \sup_{1\leq l \leq n}\gamma_l \1_{\{\tau^Y_l \leq n\}} \geq \frac{\epsilon}{2} \right] =0.
\end{align} 
Moreover, it follows from the argument in Chapter 7 of \cite{Billingsley} that 
\begin{align*}
\lim_{\delta \to 0} \lim_{n\to \infty} \pr \left[\frac{1}{\textcolor{black}{\sigma}\sqrt{n}} \sup_{1\leq i < j  \leq n: |i-j| \leq n\delta }|S(i) - S(j)| \geq \frac{\epsilon}{6} \right] =0.
\end{align*} 
This fact together with \eqref{tight2} implies that condition (ii) holds. Therefore the sequence of processes $(Y_n(.))$ is tight. This conclude the proof of Corollary \ref{meander}. 
 
\section{Proof of Theorem \ref{skew}}\label{proofoftheorem}
  The proof of Theorem  \ref{skew} follows the same strategy as the one of Corollary \ref{meander}.

Let us   fix  briefly the notations.
 
  $\bullet \quad \mathcal F_n$ is the $\sigma$-field generated by the random variables $\xi_1,  \xi'_1,  \eta_1, \ldots, \xi_n,  \xi'_n,  \eta_n$ for $n \geq 1$ (as before, $\mathcal F_0= \{\emptyset, \Omega)\}$ by convention);

 $\bullet \quad S'=(S'(n))_{n \geq 0}$ denotes the  random walks with steps $\xi'_k$;

 $\bullet\quad (\ell'_l)_{l \geq 0}$ is the sequence of  ascending ladder epochs    of the random walk $S'$ defined inductively by $\ell'_0 =0$ and, for any $l \geq 1$,
\[
\ell'_{l+1}  = \min\{n>\ell_l\mid S'(n) \textcolor{black}{>} S'{(\ell_l) }\};
\]
	
 $\bullet\quad h'$ denotes the ``ascending renewal function''
	of $S'$ defined by
	$$h'(x) =  \textcolor{black}{1+ \displaystyle \sum_{l=1}^{+\infty} \pr[S'(\ell'_l) \leq  x] } \quad  \text{ if } \quad x \geq 0 \quad  \text{and} \quad h'(x)=0 \quad     \text{otherwise.}$$
The function $h'$ is  increasing  and $h'(x) = O(x)$ as $x \to \infty$. 
 
  $\bullet\quad (\tau^X_l)_{l \geq 0}$  is the sequence  of stopping times with respect to the filtration $(\mathcal F_n)_{n \geq 0}$ defined by 
  	\[\tau^X_0 = 0 \quad \text{ and}\quad \tau^X_{l+1} = \inf \{ n > \tau^X_l: X(n) =  0\}\quad \text{ for \ any} \quad l\geq 0;
	\]

$\bullet \quad$ for $ n>k\geq 1$,  
\[
\begin{cases}
 U_{k, n}&= \eta_{k+1}+ \xi_{k+2}  + \ldots + \xi_n \  \text{when} \ \eta_{k+1}> 0;\\ 
U'_{k, n}&= \eta_{k+1}+ \xi'_{k+2}  + \ldots + \xi'_n \  \text{when} \ \eta_{k+1}< 0.
\end{cases}
 \]

\subsection{On the fluctuation of the perturbed random walk $X$}

\begin{lemma} \label{LemB2} It holds, as $n \to +\infty, $
	\begin{equation} \label{tauY>n}
	\pr[\tau^X_1>n] \sim  \frac{ c_1 \bE[h(\gamma_1)\1_{\{\eta_1>0\}}] + c'_1 \bE[h'(-\eta_1)\1_{\{\eta_1<0\}}]}{\sqrt{n}}, 
	\end{equation} 
	and 
	\begin{equation}\label{tauY=n}
	\pr[\tau^X_1 = n] \sim   \frac{ c_1 \bE[h(\eta_1)\1_{\{\eta_1>0\}}] + c'_1 \bE[h'(-\eta_1)\1_{\{\eta_1<0\}}]}{2 n^{3/2}}.
	\end{equation}
	Moreover, 	
	\begin{equation} \label{LemB'2}
	\Sigma_n (= \Sigma^X):=  \sum_{l=0}^{+\infty} \pr[\tau^X_l = n] \sim  
	\frac{1}{ c_1 \bE[h(\eta_1)\1_{\{\eta_1>0\}}] + c'_1 \bE[h'(-\eta_1)\1_{\{\eta_1<0\}}]} \frac{1}{\pi \sqrt{n}}.
	\end{equation} 
\end{lemma}
\begin{proof} We write  
	\begin{align*}
	\pr[\tau^X_1 > n] &= \pr[\eta_1>0, \eta_1+\xi_2> 0, \ldots, \eta_1+\xi_2+\ldots +\xi_n  >0] \\
	&\quad + \pr[\eta_1<0, \eta_1+\xi'_2< 0, \ldots, \eta_1+\xi'_2+\ldots +\xi'_n  <0]  
	\\
	&=\sum_{x> 0}\left( \pr[\tau^S(x)>n-1]\mathbb P[\eta_1=x] + \pr[\tau^{S''}(x)>n-1]\mathbb P[\eta_1=-x] \right).
	\end{align*}
	Using a similar argument as the one in Lemma \ref{LemB}, we get \eqref{tauY>n}. The relation  \eqref{tauY=n} can be obtained by first writing 
	\begin{align*} 
	\pr[\tau^X_1=n+2] 	&= \sum_{x> 0}   \pr[\eta_1=x]  \sum_{y\leq -1}     \pr[\xi_{1}=y] 
	\sum_{z=1}^{\vert y\vert}  \pr[ \tau^S(x)>n, x+S(n)=z] \\
	& +  \sum_{x> 0}   \pr[\eta_1=-x]  \sum_{y\leq -1}     \pr[-\xi'_{1}=y] 
	\sum_{z=1}^{\vert y\vert}  \pr[ \tau^{S''}(x)>n, x+S''_n=z], 
	\end{align*}
and then repeating the argument in the proof of Lemma \ref{LemB}. Finally, using again Theorem B in \cite{Doney97}, we get \eqref{LemB'2}. 
\end{proof}

In the following we show that the one and two marginal distributions of $Y_n$ converge to the ones of a skew Brownian motion with parameter $\alpha$ defined in (\ref{def:alpha}). The convergence of higher marginal distributions can be proved by   induction   with  similar arguments. 
\textcolor{black}{Morover, we suppose that $\sigma = \sigma' = 1$ in order to ease our notations. The proof for general $\sigma$ and $\sigma'$ follows the same line.}

\subsection{One-dimensional distribution}

Note that   
\begin{align*} 
\alpha = \lim_{n\to \infty} \frac{\pr[\tau^X_1 >n, \eta_1>0]}{\pr[\tau^X_1 > n]}. 
\end{align*} 
For a function $\varphi$, we denote $\hat{\varphi}(u) = \alpha \varphi(u) + (1-\alpha)\varphi(-u)$. 
\begin{lemma}\label{1-dim2}
	For any $t\in [0, 1]$, it holds
	$$\lim_{n\to +\infty}\bE \left[ \varphi \left( X_n(t)\right)\right] 
		= \int_0^{+\infty} \hat{\varphi}(u) \frac{2e^{-u^2/2t}}{\sqrt{2\pi t}}du = \int_{-\infty}^{+\infty} \varphi(u)p^\alpha_t(0,u)du.$$
\end{lemma}
\begin{proof}
	We fix $ x\geq 0$ and $t\in (0, 1)$. Follows the argument in the proof of Lemma \ref{1-dim}, we write $\bE \left[ \varphi \left( \frac{X({[nt]})}{\sqrt{n}}\right)\right] = A_{0+}(n)+ A_{0-}(n)$, where 
	\begin{align*}
	A_{0+}(n)&= \sum_{k=0}^{[nt]} \sum_{l=0}^{+\infty}  \bE \left[ \varphi \left( \frac{X({[nt]})}{\sqrt{n}}\right); \tau^X_l = k, X(k+1) > 0, \ldots, X({[nt]}) > 0 \right],\\
	A_{0-}(n)& = \sum_{k=0}^{[nt]} \sum_{l=0}^{+\infty}  \bE \left[ \varphi \left( \frac{X({[nt]})}{\sqrt{n}}\right); \tau^X_l = k, X(k+1) < 0, \ldots, X({[nt]}) < 0 \right].
	\end{align*}
	We have 
	\begin{align*}
	A_{0-}(n)&\approx \sum_{k=0}^{[nt]-1}  \Sigma_k \bE \left[ \bE\left( \varphi \left (- \frac{x + S''_{{[nt]}-k-1}}{\sqrt{n}}\right)\Big| \tau^{S''}(x) > {[nt]}-k-1\right) \right. \\
	&\qquad \qquad \qquad  \qquad \qquad \qquad \times \left. \pr[\tau^{S''}(x) > {[nt]}-k-1] \Big|_{  \eta_1 =-x<0}\right].
	\end{align*}
	Using Lemmas \ref{LemA}, \ref{LemC} and  \ref{LemB2} and the dominated convergence theorem as in the proof of Lemma \ref{1-dim}, we get 	
	\begin{align*}
	\lim_{n\to +\infty} A_{0-}(n) 
	&= (1-\alpha)  \ \int_0^t \frac{1}{\pi \sqrt{s(t-s)}} \left(\int_0^{+\infty} \varphi(-z\sqrt{t-s})ze^{-z^2/2}dz \right)ds\\ 
	&= \alpha  \int_0^{+\infty} \varphi(-u) \frac{2e^{-u^2/2t}}{\sqrt{2\pi t}}du. 
	\end{align*}
	A similar computation yields
	\begin{align*}
	\lim_{n\to +\infty} A_{0+}(n) 
	&= \alpha \int_0^{+\infty} \varphi(u) \frac{2e^{-u^2/2t}}{\sqrt{2\pi t}}du. 
	\end{align*} 
	Therefore, 
		$$\lim_{n\to +\infty}\bE \left[ \varphi \left( \frac{X({[nt]})}{\sqrt{n}}\right)\right] 
		= \int_0^{+\infty} \hat{\varphi}(u) \frac{2e^{-u^2/2t}}{\sqrt{2\pi t}}du.$$
	Using \eqref{skewdensity} and an argument similar to \eqref{lagtime}, we conclude  Lemma \ref{1-dim2}. 
\end{proof}


\subsection{Finite dimensional distribution}
The convergence of finite-dimensional distributions of $(X_n(t))_{n \geq 1}$ is more delicate. We detail the argument for two-dimensional ones.

Let us   fix $0<s<t, n \geq 1$ and  denote 
$$\kappa = \kappa(n, s) = \min\{k>[ns]: X(k) = 0\}.$$
We write 
\begin{align*}
&\bE\left[ \varphi_1\left( \frac{X({[ns]})}{\sqrt{n}}\right) \varphi_2 \left( \frac{X([nt])}{\sqrt{n}}\right)\right] = A_1(n) + A_2(n),
\end{align*}
where 
\begin{align*}
A_1(n) &= \sum_{k = [ns]+1}^{[nt]} \bE\left[ \varphi_1\left( \frac{X({[ns]})}{\sqrt{n}}\right) \varphi_2 \left( \frac{X([nt])}{\sqrt{n}}\right)\1_{\{\kappa = k\}} \right],\\
A_2 (n) &= \bE\left[ \varphi_1\left( \frac{X({[ns]})}{\sqrt{n}}\right) \varphi_2 \left( \frac{X([nt])}{\sqrt{n}}\right)\1_{\{\kappa > [nt]\}} \right].
\end{align*}

\subsubsection{Estimate of $A_1(n)$}

As in the previous section, based on the sign of $X([ns])$,  $A_1(n)$ can be decomposed as $A_{1+}(n) + A_{1-}(n)$ where    
\begin{align*}
A_{1+}(n) 
&= \sum_{k_1=0}^{[ns]}  \Sigma_{k_1} \sum_{k_2 = [ns]+1}^{[nt]}   \bE \left[\varphi_2 \left( \frac{X([nt]-k_2)}{\sqrt{n}}\right)\right]\\
&\times \bE\left[ \bE\left[ \varphi_1\left( \frac{x+ S([ns]-k_1-1)}{\sqrt{n}}\right);   \tau^S(x) = k_2-k_1 -1  \right]\Big|_{\eta_1 = x>0}\right],\\
A_{1-}(n) 
& =  \sum_{k_1=0}^{[ns]}  \Sigma_{k_1} \sum_{k_2 = [ns]+1}^{[nt]} \bE \left[\varphi_2 \left( \frac{X([nt]-k_2)}{\sqrt{n}}\right)\right] \\
&\times \bE\left[ \bE\left[ \varphi_1\left(- \frac{-x+ S''_{[ns]-k_1-1}}{\sqrt{n}}\right);   \tau^{S''}(-x) = k_2-k_1 -1  \right]\Big|_{\eta_1 = x<0}\right].
\end{align*}
Using Lemmas \ref{LemB2}, \ref{1-dim2} and repeating the argument in Section \ref{Sec:3.2.1}  we get 
\begin{align*}
\lim_{n\to \infty}A_{1+}(n) &= \frac{2\alpha}{\pi \sqrt{ s (t-s)}}  \int_0^{+\infty} \int_0^{+\infty} \varphi_1(v) \hat{\varphi}_2(u) e^{-v^2/2s}e^{ - \frac{(u+v)^2}{2(t-s)}} dudv,\\
\lim_{n\to \infty}A_{1-}(n)  &= \frac{2(1-\alpha)}{\pi \sqrt{ s (t-s)}}  \int_0^{+\infty} \int_0^{+\infty} \varphi_1(-v) \hat{\varphi}_2(u) e^{-v^2/2s} e^{ - \frac{(u+v)^2}{2(t-s)}} dudv.
\end{align*}
Therefore, 
\begin{align} \label{A1alpha} 
\lim_{n\to \infty} A_{1}(n) = 4  \int_0^{+\infty} \int_0^{+\infty} \hat{\varphi}_1(v) \hat{\varphi}_2(u) p_s(v) p_{t-s}(u+v)  dudv.
\end{align}

\subsubsection{Estimate of $A_2 (n)$}
We decompose $A_2 (n) = A_{2+}(n) + A_{2-}(n)$, where  
\begin{align*}
A_{2+} (n)=&  \sum_{k \leq  [ns]} \sum_{l\geq 0}     \bE\Big[ \varphi_1\left( \frac{X({ [ns]})}{\sqrt{n}}\right) \varphi_2 \left( \frac{X([nt])}{\sqrt{n}}\right);   \\
& \qquad \qquad  \qquad \qquad \tau^X_{l} = k, \eta_{k+1} > 0, U_{k, k+2} >0, \ldots, U_{k, [nt]} > 0\Big],  \\
A_{2-}(n) = &   \sum_{k \leq  [ns]} \sum_{l\geq 0}   \bE\Big[ \varphi_1\left( \frac{X({ [ns]})}{\sqrt{n}}\right) \varphi_2 \left( \frac{X([nt])}{\sqrt{n}}\right);  \\
& \qquad \qquad  \qquad \qquad \tau^X_{l} = k, \eta_{k+1} < 0, U'_{k, k+2} <0, \ldots, U'_{k, [nt]} < 0\Big].
\end{align*}
We have 
\begin{align*}
A_{2-}(n) &= \bE\Big[ \varphi_1\left( \frac{U'_{k,  [ns]}}{\sqrt{n}}\right) \varphi_2 \left( \frac{U'_{k,[nt]}}{\sqrt{n}}\right);   \tau^X_{l} = k, U'_{k, k+2} <0, \ldots, U'_{k, [nt]} < 0\Big]  \\
&=    \sum_{k \leq  [ns]} \Sigma_k \bE\Bigg[ \bE\Big[ \varphi_1\left(- \frac{-x + S''_{  [ns]-k-1}}{\sqrt{n}}\right) \varphi_2 \left(- \frac{-x + S''_{[nt]-k-1}}{\sqrt{n}}\right);   \\
& \qquad \qquad \qquad \qquad\qquad \qquad \quad \qquad \qquad  \tau^{S''}(x) > [nt]-k-1\Big]\Bigg|_{ \eta_1=x<0}\Bigg].
\end{align*}
By adapting the argument in Section \ref{Sec:3.2.2} we get 
\begin{align*}
\lim_{n\to +\infty} A_{2-} (n) =  {1- \alpha \over\pi \sqrt{  s (t-s)}}  \int_0^{+\infty} &dy' \int_0^{+\infty} dz' \quad\varphi_1(-y')\varphi_2(-z')  \\
& \times e^{-y'^2/2s}  
\left(e^{-\frac{(z'-y')^2}{2(t-s)}} - e^{-\frac{(z'+y')^2}{2(t-s)}} \right).
\end{align*}
Similar argument yields, 
\begin{align*} 
\lim_{n\to +\infty} A_{2+} (n) 
=   {\alpha \over\pi \sqrt{  s (t-s)}}  \int_0^{+\infty} & dy' \int_0^{+\infty} dz' \varphi_1(y')\varphi_2(z')   \\
& \times  e^{-y'^2/2s} 
\left(e^{-\frac{(z'-y')^2}{2(t-s)}} - e^{-\frac{(z'+y')^2}{2(t-s)}} \right).
\end{align*} 
Therefore, 
\begin{align} \label{A2alpha}
\lim_{n\to +\infty} A_{2} (n) 
=  2  \int_0^{+\infty}  \int_0^{+\infty}   & (\alpha \varphi_1(y') \varphi_2(z') +  (1-\alpha) \varphi_1(-y')\varphi_2(-z')) p_s(y')  \notag \\
& \times (p_{t-s}(y'-z') - p_{t-s}(y'+z') )dy'dz'.
\end{align}

Finally, it follows from \eqref{A1alpha},  \eqref{A2alpha} and \eqref{skewdensity} that for all $0 \leq s < t \leq 1$, it holds that 
\begin{align*}
\lim_{n \to \infty} \bE\left[ \varphi_1\left( X_n(s)\right) \varphi_2 \left( X_n(t)\right)\right] &=
\lim_{n \to \infty} \bE\left[ \varphi_1\left( \frac{X({[ns]})}{\sqrt{n}}\right) \varphi_2 \left( \frac{X([nt])}{\sqrt{n}}\right)\right] \\
&= \bE \left[ \varphi_1(B^\alpha_s) \varphi_2(B^\alpha_t)\right],
\end{align*}
where $B^\alpha$ is a skew Brownian motion with parameter $\alpha$. 

\subsection{Tightness}
Note that the modulus of continuity $w_{X_n}$ of $X_n$ satisfies
\begin{align*}
w_{X_n}(\delta)  \leq \frac{3}{\sqrt{n}}\sup_{1\leq  i <  j \leq n; |i-j| \leq n\delta}|S(i) - S(j)| & +\frac{3}{\sqrt{n}}\sup_{1\leq  i <  j \leq n; |i-j| \leq n\delta}|S'(i) - S'(j)| \\
&+ \frac{2}{\sqrt{n}}\sup_{1\leq l \leq n}|\eta_l| \1_{\{\tau^X_l \leq n\}}.
\end{align*}
Therefore,  the tightness of $X$ can be proved in a similar way as the one of $Y$. 
We achieve the proof of Theorem \ref{skew} using Theorem 7.1 in \cite{Billingsley}.

\section*{Acknowledgment}
H.-L. Ngo thanks the University of Tours for generous hospitality in the Instittue Denis Poisson (IDP) and financial support in May  2019. 
M. Peign\'e  thanks  the  Vietnam Institute for Advanced Studies in Mathematics (VIASM) and the VietNam Academy of Sciences And Technology  (VAST)  in Ha Noi for  their kind and friendly hospitality and accommodation in June 2018.

\end{document}